\documentclass[12pt]{article}

\usepackage[usenames]{xcolor}
\usepackage{graphicx}
\usepackage[colorlinks=true,
linkcolor=webgreen,
filecolor=webbrown,
citecolor=webgreen]{hyperref}
\definecolor{webgreen}{rgb}{0,.5,0}
\definecolor{webbrown}{rgb}{.6,0,0}
\usepackage{color}

\usepackage{amssymb}
\usepackage{amsthm}
\usepackage{mathtools}
\usepackage{algpseudocode,algorithm,algorithmicx}
\usepackage{bm}

\usepackage{tikz}
\usepackage{subcaption}
\usepackage{verbatim}
\usepackage{diagbox}
\usepackage{float}

\usepackage{authblk}

\theoremstyle{plain}
\newtheorem{theorem}{Theorem}
\newtheorem{corollary}[theorem]{Corollary}
\newtheorem{lemma}[theorem]{Lemma}

\theoremstyle{definition}
\newtheorem{definition}[theorem]{Definition}

\theoremstyle{remark}

\newcommand{\floor}[1]{\left\lfloor#1\right\rfloor}
\newcommand{\ceil}[1]{\left\lceil#1\right\rceil}

\makeatletter
\renewcommand{\Function}[2]{%
	\csname ALG@cmd@\ALG@L @Function\endcsname{#1}{#2}%
	\def\curfunc{#1}%
}
\newcommand{\funclabel}[1]{%
	\@bsphack
	\protected@write\@auxout{}{%
		\string\newlabel{#1}{{\curfunc}{\thepage}}%
	}%
	\@esphack
}

\begin{document}
\author{Laura Monroe}
\affil{Ultrascale Systems Research Center, Los Alamos National Laboratory, Los Alamos, NM 87501
\newline
lmonroe@lanl.gov}
\renewcommand\Affilfont{\itshape\small}
\renewcommand\footnotemark{}
	\thanks{
		This publication has been assigned the LANL identifier LA-UR-21-25242. 
		 
		\hspace*{0.82em}This work has been authored by an employee of Triad National Security, LLC, operator of the Los Alamos National Laboratory under Contract No.89233218CNA000001 with the U.S. Department of Energy. This work was also supported by LANL's Ultrascale Systems Research Center at the New Mexico Consortium (Contract No. DE-FC02-06ER25750).).
		The United States Government retains and the publisher, by accepting this work for publication, acknowledges that the United States Government retains a nonexclusive, paid-up, irrevocable, world-wide license to publish or reproduce this work, or allow others to do so for United States Government purposes.  
	} 
\date{\vspace{-5ex}}
\title{Binary Signed-Digit Integers and the Stern Diatomic Sequence
}

\maketitle

\begin{abstract}
	Stern's diatomic sequence is a well-studied and simply defined sequence with many fascinating characteristics. The binary signed-digit representation of integers is an alternative representation of integers with much use in efficient computation, coding theory and cryptography.  

	We link these two ideas here,  showing that the number of $i$-bit binary signed-digit representations of an integer $n$ with $n<2^i$ is the $(2^i-n)^\text{th}$ element in Stern's diatomic sequence. This correspondence makes the vast range of results known for Stern's diatomic sequence available for consideration in the study of binary signed-digit integers. 
\end{abstract}

\section{Introduction}
Integers may be represented in binary signed-digit (BSD) representation, in which each integer is represented in terms of sums or differences of powers of $2$. This is in contrast to binary representation, in which only sums are allowed. 

BSD representations of an integer are not unique. In fact, there are an infinite number of such representations for any non-$0$ integer $n$ using an arbitrary number of bits, and of course, a finite number for any fixed number of bits.  The number of such representations is of interest.

In this paper, we show a correspondence between the number of BSD representations of an integer and Stern's diatomic sequence. 
We first do this directly, using a new recurrence relation for the number of ways to express a non-negative integer in BSD form. 

We also show a direct translation between an $i$-bit BSD representation of an integer $n$ and a hyperbinary representation of the related integer $2^i-1-n$. This gives a second independent proof of the correspondence between the BSD representations and the Stern sequence. 

Finally, we give a simple $\mathcal{O}(\log(n))$ algorithm for calculating the number of BSD representations of $n$, based on an algorithm for calculating the elements of Stern's diatomic sequence.

\section{Number of BSD representations of an integer}
Signed-digit representations of integers have been discussed for almost 300 years, starting with Colson in 1726~\cite{colson1726} and Cauchy in 1840~\cite{cauchy1840}.
Binary signed-digit representation is of great use in computer science for the efficient operations it affords, and has been presented as such since early in the history of modern computing by Shannon~\cite{shannon50}, Booth~\cite{booth1951} and others. This representation has many applications in basic calculation \cite{avizienis1961signed}, coding theory, and elsewhere \cite{morain90,egecioglu90,koblitz91}. A general discussion can be found in \cite{Shallit92aprimer}.

\begin{definition}[BSD representation]\label{def_bsd}
	An integer $n$ is in \emph{BSD representation} when
	$$
		n=\sum_{j=0}^{i-1}b_j2^j \text{, where $i\ge\ceil{\log_2(n)}$, and $b_j \in \{1,0,-1\}$.}
    $$
\end{definition}
Throughout this paper, $f(n,i)$ denotes the number of ways to represent the integer $n$ in BSD form on $i$ bits.

There are $3^i$ different BSD representations of integers on $i$ signed bits. However, these represent only $(2^{i+1}-1)$ integers, since the maximum such integer is $(2^i-1)$ and the minimum is $(-2^i+1)$. There are thus multiple BSD representations for at least some of the $i$-bit BSD integers. 

In this section, we develop and discuss several basic identities on the number of BSD representations of integers $n$, several of which are already known and discussed by Ebeid and Hasan in \cite{ebeid07} and elsewhere.  

These identities will be used in the proof of the correspondence between the sequence $f(n,i)$  and Stern's diatomic sequence, and in the algorithm for calculating the number of BSD representations of an integer $n$ on $i$ bits.

\begin{lemma} \cite{ebeid07}
\label{symmetricity}
$f(n,i) = f(-n,i)$. 
\end{lemma}
\begin{proof}
The representations of $-n$ are obtained by multiplying each bit in the BSD representation of $n$ by $-1$. 
\end{proof}

\begin{theorem} 
\label{theorem_bsd_recursive}
The number of ways to express a non-negative integer $n$ in BSD form on $i \ge 0$ bits is expressed by the recurrence relation
$$
f(n,i) = f(|n-2^{i-1}|, i-1) + f(n,i-1)
\text{, where } f(0,i) =  1.
$$
\end{theorem}
\begin{proof}
Let $n\ge 0$. Since $n$ is non-negative, the value of the $(i-1)$ position can not be $-1$. The BSD representations of $n$ with $i$ bits can then be partitioned into those with a 1 in the $(i-1)$ place, and those with a 0 in the $(i-1)$ place. 

The set of representations with 1 in the $(i-1)$ place has a one-to-one correspondence with the representations of $n-2^{i-1}$ of length $i-1$. This number is $f(n-2^{i-1}, i-1)$, and by Lemma~\ref{symmetricity}, this is $f(|n-2^{i-1}|, i-1)$.

The set of representations with 0 in the $(i-1)$ place has a one-to-one correspondence with the set of representations expressing $n$ of length $i-1$. There are  $f(n,i-1)$ of these.
\end{proof}
The following are some properties of $f(n,i)$ on non-negative integers $n$ following directly from Theorem~\ref{theorem_bsd_recursive}. These are used throughout this paper. Equation~\ref{bsd_Equation_2} is shown by Ebeid and Hasan in \cite{ebeid07}. 
\begin{align}
    &f(n,i) = 0\text{, when } n \ge 2^i.\label{bsd_Equation_1}\\
    &f(0,i) = 1\text{, for all }i.\label{bsd_Equation_2}\\
    &f(2^k \cdot n ,i) = f(n,i-k).\label{bsd_Equation_3}
\end{align}
Corollary~\ref{props_fni2} follows from Theorem~\ref{theorem_bsd_recursive} and is also shown by Ebeid and Hasan in \cite{ebeid07}.
\begin{corollary}\cite{ebeid07}\label{props_fni2} Let $n$ be a non-negative integer. The following are properties of $f(n,i)$:
	\begin{align*}
	f(2n,i) &= f(n,i-1).\\
	f(2n+1,i) &=f(n,i-1)+f(n+1,i-1).
	\end{align*}
\end{corollary}
\section{BSD representations and Stern's diatomic sequence}
Stern's diatomic sequence, or the Stern-Brocot sequence, is a well-known integer sequence first discussed by M.A. Stern in 1858 \cite{Stern1858}. Stern's diatomic sequence has been studied in detail for some 160 years, and much is known about it. Some of these details may be found in \cite{calkin00,carlitz1964,dijkstra82,lehmer29,lind1969,shannon50,Stern1858}. In particular, Lehmer provides a discussion in \cite{lehmer29}, Northshield gives a good overview in  \cite{northshield2010stern}, and the Online Encyclopedia of Integer Sequences entry A002487 \cite{oeis-web} has many references for this sequence. 

We show in this section that the number of BSD representations of $n$ with $n<2^i$ is the $(2^i-n)^{\text{th}}$ entry in Stern's diatomic sequence. 
The extensive set of identities concerning Stern's diatomic sequence may be used to great effect to understand the number of BSD representations of an integer $n$ on $i$ bits.
 
Stern's diatomic sequence  is defined recursively:
\begin{definition} [Stern's diatomic sequence]\cite{Stern1858} \label{def_sd}
Let $n$ be a non-negative integer. \emph{Stern's diatomic sequence} $c(n)$ is defined as follows:
	$$
		c(n)= 
		\begin{cases}
		0 & \text{if $n=0$,} \\
		1 & \text{if $n=1$,} \\
		\text{c($m$)} & \text{if $n=2m$,} \\
		\text{c($m$)+c($m+1$)} & \text{if $n=2m+1$.} 
		\end{cases}
	$$ 
\end{definition}

The sequence has many interesting properties, a few of which we list here: 
\begin{list}{$\cdot$}{}  
	\item ${c(n)}/{c(n+1)}$ hits each positive rational once, each in lowest terms \cite{Stern1858,calkin00}. These fractions were presented in tree form by Calkin and Wilf in \cite{calkin00}. This tree was generalized by Bates and Mansour in \cite{bates2011}.
	\item $c(n+1)$ is the number of hyperbinary representations of $n$ \cite{carlitz1964,lind1969,reznick90}.
	\item $c(n+1)$ is the number of alternating bit sets in $n$  \cite{finch2003mathematical}. 
	\item $c(n)$ is the Dijkstra \emph{fusc} function  \cite{dijkstra82}. 
\end{list}
Theorem~\ref{theorem_stern_bsd} gives a direct proof that the number of ways positive integers may be expressed in $i$-bit BSD representation is Stern's diatomic sequence up to $2^i$, in reverse. This is shown by exploiting the similarities between Corollary~\ref{props_fni2} and the definition of Stern's diatomic sequence in Definition~\ref{def_sd}. This is the main result of this paper.

\begin{theorem} 
	\label{theorem_stern_bsd}
	Let $f(n,i)$ be the number of ways to express the integer $n$ in BSD representation on $i$ bits, where $0<n<2^i$. Then
	$$
		f(n,i) = c(2^i-n).
	$$
\end{theorem}
\begin{proof}We proceed by induction on $i$. 
\\\emph{Case 1}: $n=2m$ is even.
	\begin{align*}
	f(2m,i)&= f(m,i-1)& \text{by Corollary ~\ref{props_fni2}}&\\
	&= c( 2^{i-1}-m )& \text{by induction}\\
	&= c( 2\cdot(2^{i-1}-m) )& \text{by Definition~\ref{def_sd}}\\
	&= c(2^i-n).
	\end{align*}
\emph{Case 2}: $n=2m+1$ is odd.
	\begin{align*}
	f(2m+1,i) &= f(2m,i) + f(2(m+1),i)& \text{by Corollary~\ref{props_fni2}}&\\ 
	&= f(m,i-1) + f(m+1,i-1)& \text{by Corollary~\ref{props_fni2}}\\ 
	&= c(2^{i-1}-m) + c(2^{i-1}-(m+1))& \text{by  induction}\\ 
	&= c(2^{i-1}-(m+1)+1) + c(2^{i-1}-(m+1))\\ 
	&= c(2^{i-1}-(m+1)) + c((2^{i-1}-(m+1))+1)\\
	&= c(2\cdot (2^{i-1}-(m+1))+1)& \text{by Definition~\ref{def_sd}}\\
	&= c(2^i-(2m+1))\\
	&= c(2^i-n).
	\end{align*}
\end{proof}
Theorem \ref{theorem_stern_bsd} shows that the sequence $f(n,i)$ is comprised of the first $2^i+1$ entries of Stern's diatomic sequence in reverse when $n=\{0,\dots,2^i\}$, and is comprised of the first $2^i+1$ entries of Stern's diatomic sequence when $n=\{-2^i,\dots,0\}$. 

The following corollary follows directly from Theorem \ref{theorem_stern_bsd}.
\begin{corollary} 
	\label{Bitflips_distance}
	Let $0<n<2^{k}$, where $k=\lceil{\log_2(n)}\rceil<i$.  Then 
	$$
	f(2^i - n,i) = c(n) = f(2^k-n,k).
	$$
\end{corollary}
The following are known properties of Stern's diatomic sequence, proven as such by Carlitz  \cite{carlitz1964} and others, and are now seen to be properties of $f(n,i)$. They also follow quickly from Theorem~\ref{theorem_bsd_recursive}. We use them throughout Section~\ref{section_alg}, in developing the algorithm for calculating the number of BSD representations of $n$. 
\begin{align}
    &f({2^i-1},i) =c(1)= 1.\label{bsd_Equation_4}\\
    &f(1,i) =c(2^i-1) = i.\label{bsd_Equation_5}\\
    &f({2^{i-1}+1},i) = c(2^{i-1}-1) = i-1.\label{bsd_Equation_6}
\end{align}
Many more identities on Stern's diatomic sequence exist in the literature and may be used in the investigation of BSD representations of integers. 

We restate a known bound on the number of BSD representations of a given bitlength, proved by Lind \cite{lind1969} and discussed by Ebeid and Hasan \cite{ebeid07} and Northshield \cite{northshield2010stern}. 
\begin{definition}[Fibonacci sequence]\label{def_fib}
    The \emph{Fibonacci sequence} is defined by: $$F_0=0 \text{; } F_1=1 \text{; } F_{n}=F_{n-1}+F_{n-2}.$$ 
\end{definition}

\begin{theorem}\cite{lind1969,ebeid07,northshield2010stern}
	The maximum number of BSD representations of any integer of a given bitlength $i$ is $F_i$, the $i^{th}$ Fibonacci number. The bitlength-$i$ integers having this maximum number of BSD representations are
	$$
	\frac{2^{i-1}+(-1)^i}{3}
	\text{ and }
	\frac{2^{i-1}-(-1)^i}{3}.
	$$
\end{theorem}
It is clear from the definition why the Fibonacci sequence occurs in the Stern sequence: if $f(n)=F_i$ and $f(n \pm 1) = F_{i-1}$, then $f(n + (n\pm 1))=F_{i+1}$. Fibonacci numbers appear often in the study of BSD representations  \cite{grabner06,Tuma_2015}. 
\section{BSD representations and hyperbinary representations}
The hyperbinary representation of an integer is another ternary representation with a base of $2$, but over $\{0,1,2\}$ rather than $\{1,0,-1\}$, as in the BSD case. 
In this section, we show an explicit one-to-one translation function taking BSD representations of non-negative integers $n$ to hyperbinary representations of $(2^i-1-n)$. 

A short independent proof of Theorem~\ref{theorem_stern_bsd} follows from this translation, calling upon the well-known identification of the hyperbinary sequence with Stern's diatomic sequence.
The translation function of Theorem~\ref{theorem_bsd_hb_rep} may also be used to exploit known properties of hyperbinary representations, and apply them to BSD representations.
\begin{definition}[Hyperbinary representation of an integer]
	An integer $n$ is in \emph{hyperbinary representation} when
	$$
		n=\sum_{j=0}^{i-1}h_j2^j \text{, where $i\ge\ceil{\log_2(n)}$, and $h_j \in \{0,1,2\}$.}
	$$
\end{definition}
\begin{definition}[Hyperbinary sequence]
	The \emph{hyperbinary sequence} $h(n)$ is the number of ways $n$ may be  expressed in hyperbinary representation.
\end{definition}
There is a simple correspondence between BSD and hyperbinary representations. They come in pairs: subtracting a BSD representation of $n$ componentwise from the $i$-bit all-$1$ vector gives a hyperbinary representation of $2^i-1-n$, and vice versa. We formalize this observation in Theorem~\ref{theorem_bsd_hb_rep} and  Corollary~\ref{cor_1-1_bsd_hyper}. 
\begin{theorem}\label{theorem_bsd_hb_rep}
    $(b_{i-1}\cdots b_0)$ is an $i$-bit BSD representation of a non-negative integer $n$ if and only if $((1-b_{i-1})\cdots (1-b_0))$ is an $i$-bit hyperbinary representation of $2^i-1-n$.
\end{theorem}
\begin{proof}
    $(b_{i-1}\cdots b_0)$ represents $n$ as a sum of $i$ powers of 2 if and only if $((1-b_{i-1})\cdots (1-b_0))$ represents $2^{i}-1-n$ as a sum of $i$ powers of 2, and
    $b_j \in \{1,0, -1\}$ if and only if $(1-b_j) \in \{0,1, 2\}$. The theorem follows.
\end{proof}
\begin{corollary}\label{cor_1-1_bsd_hyper}
    There is a one-to-one correspondence between the $i$-bit BSD representations of $n$ and the $i$-bit hyperbinary representations of $(2^i-1-n)$, so $$f(n,i) = h(2^i-1-n).$$
\end{corollary}
The translation function given in Theorem~\ref{theorem_bsd_hb_rep} can be expressed as a direct translation between the digits of a BSD representation $(b_{i-1}\cdots b_0)$ of an integer $n$ and the digits of the corresponding hyperbinary representation, $(h_{i-1}\cdots h_0)$ of $2^i-1-n$, given by
\begin{equation}\label{eq_translation}
\begin{aligned}
   1\leftrightarrow 0&,\\ 0\leftrightarrow 1&,\\
  -1\leftrightarrow 2&. 
\end{aligned}
\end{equation}
The already-known Theorem \ref{theorem_stern_hyperbinary_orig} identifies the hyperbinary sequence with Stern's diatomic sequence. It was shown by Carlitz in \cite{carlitz1964}, Lind in \cite{lind1969}, and Reznick in \cite{reznick90}. In \cite{stanley2010}, Stanley and Wilf further refine this, and derive formulas for the partitions of $n$ into powers of 2 in which exactly $k$ parts have multiplicity $2$.

\begin{theorem}\label{theorem_stern_hyperbinary_orig}\cite{carlitz1964,lind1969,reznick90} The hyperbinary sequence is Stern's diatomic sequence, offset by 1: $$h(n)=c(n+1).$$
\end{theorem}
From the correspondence between BSD and hyperbinary representations in Corollary~\ref{cor_1-1_bsd_hyper}, and the previously known identification of the hyperbinary sequence with the Stern sequence in Theorem~\ref{theorem_stern_hyperbinary_orig}, we immediately obtain a simple and independent proof of the result in Theorem~\ref{theorem_stern_bsd}, restated here.
\begin{corollary}[Theorem~\ref{theorem_stern_bsd}] 
	\label{bf_stern_via_h}
$
f(n,i) = h(2^i-1-n) = c(2^i-n).
$
\end{corollary}

\section{An  algorithm for counting BSD representations of $n$}\label{section_alg}
In this section, we use Equations~\ref{bsd_Equation_4}-~\ref{bsd_Equation_6}, which follow from Theorem~\ref{theorem_stern_bsd}, to derive a simple and direct $\mathcal{O}(i)$ algorithm for calculating the number of $i$-bit BSD representations of $n$. This algorithm is equivalent to calculating the $r^{\text{th}}$ element in Stern's diatomic sequence, where $i=\ceil{\log_2(r)}$, and $n=2^i-r$. A refinement of this algorithm then leads to the simple  Algorithm~\ref{sb_alg}, calculating the number of $i$-bit BSD representations of $n$. 
\begin{algorithm}[H]\caption{Number of BSD representations on $i$ bits for $n$}\label{sb_alg}
	\begin{algorithmic}[0] 
		\Function{BSD}{$n,i$} \funclabel{sb_alg:n,i}
		\State $n \gets |m|$ \Comment Lemma \ref{symmetricity} and Equation ~\ref{bsd_Equation_3}
		\State $i \gets i-j$ \Comment Equation ~\ref{bsd_Equation_3}
		\State $num\_reps \gets 1$
		\If{$n<2^i$}
		\State $lower\gets 1; upper\gets 0 $ \Comment Equations ~\ref{bsd_Equation_1} and  ~\ref{bsd_Equation_2}
		\For{$\ell \gets 1, i$} 
		\State $num\_reps \gets lower+upper$ \Comment Theorem \ref{repeatedsum}
		\If{$n_{i-\ell}=0$}
		\State $upper\gets num\_reps $		
		\ElsIf {$n_{i-\ell}=1$}
		\State $lower\gets num\_reps $		
		\EndIf
		\EndFor
		\EndIf
		\State \textbf{return} $num\_reps$
		\EndFunction
	\end{algorithmic}
\end{algorithm}

Theorem \ref{repeatedsum} gives $f(n,i)$ for calculating the number of  $i$-bit BSD  representations for any integer $n$, and gives rise to Algorithm \ref{sb_alg}, which is $\mathcal{O}(i)$,  and is $\mathcal{O}(\log(n))$ if $i=\ceil{\log_2(n)}$. 
This theorem is a quick  consequence of the definition of BSDs.   
\begin{theorem} 
	\label{repeatedsum}
	Let $n=2^j\cdot m$, $m$ odd, and let $j<i$.
	Then $$f(n,i) = f(n-2^j,i)+ f(n+2^j,i).$$
\end{theorem}
\begin{proof}
	\begin{align*}
	f(n,i) &= f(2^jm,i) \\
	&= f(m, i-j)&\text{ by Equation ~\ref{bsd_Equation_3}}&\\
	&=  f(m-1, i-j)+ f(m+1, i-j)&\text{ by Corollary~\ref{props_fni2}}\\
	&= f(2^j(m-1),i)+ f(2^j(m+1),i)&\text{ by Equation~\ref{bsd_Equation_3}}\\
	&= f(n-2^j,i)+ f(n+2^j,i).
	\end{align*}
\end{proof}
We derive Algorithm \ref{sb_alg} in this manner. Let  $n=2^jm$ be an integer with $m$ odd; $j$ and thus $m$ can be calculated bitwise in $\mathcal{O}(\log(n))$. Let $-2^i<n<2^i$, so $i\ge\floor{\log_2(n)}$. We calculate repeatedly on intervals starting with $[0,2^i]$, and apply Theorem \ref{repeatedsum} to midpoints, replacing either the upper or lower bound on the interval with the midpoint, until we reach the desired value for $m$.

The following two lemmas and Theorem~\ref{Bitflips_arithmetic} refine the $\mathcal{O}(\log(i))$ Algorithm \ref{sb_alg} to the  $\mathcal{O}(\log(n))$ Algorithm \ref{sb_alg_log}. 
\begin{lemma}
	\label{Bitflips_2iless1_closed}
	$f(2^k-1,i) = 1+(i-k)k$, when $i\ge{k}$.
\end{lemma}
\begin{proof}
	Proof via induction on $i$. This is true for $i=k$, by Equation~\ref{bsd_Equation_5}.
	\begin{align*}
	f(2^k-1,i) &=  f(2^k-2,i) +  f(2^k,i)&\text{by Corollary~\ref{props_fni2}}&\\
	&= f(2^{k-1}-1,i-1) +  f(2^k,i)&\text{by Corollary~\ref{props_fni2}}\\
	&= f(2^{k-1}-1,i-1) +  (i-k)&\text{by Equations~\ref{bsd_Equation_3} and \ref{bsd_Equation_6}}\\
	&= 1+((i-1)-(k-1))(k-1) +  (i-k)&\text{by  induction}\\
	&= 1+((i-k)k.
	\end{align*}
\end{proof}
\begin{lemma}
	\label{Bitflips_2i-1plus1_closed}
	$f(2^{k-1}+1,i) = (k-1)+(i-k)k$, when $i\ge{k}$.
\end{lemma}
\begin{proof} 
	Proof via induction on $i$. This is true for $i=k$, by Equation~\ref{bsd_Equation_6}.
	\begin{align*}
	f(2^{k-1}+1,i) &=  f(2^{k-1},i) +  f(2^{k-1}+2,i)&\text{by Corollary~\ref{props_fni2}}&\\
	&=  (i-(k-1)) +  f(2^{k-1}+2,i)&\text{by Equations~\ref{bsd_Equation_3} and \ref{bsd_Equation_6}}\\
	&=  (i-(k-1)) +  f(2^{k-2}+1,i-1)&\text{by Corollary ~\ref{props_fni2}}\\
	&=  (i-(k-1))+(k-2)\\ &\qquad\qquad+((i-1)-(k-1))(k-1) &\text{by induction}\\
	&=  (i-1)+(i-k)(k-1) \\
	&=  (i-1)+(i-k)k-(i-k)\\
	&=  (k-1)+(i-k)k.
	\end{align*}	
\end{proof}
Theorem \ref{Bitflips_arithmetic} gives an arithmetic progression on $f(n,i)$ for any $n$. $f(n,\lceil{\log_2(n)}\rceil)$ (or $c(2^k-n)$) is the first term and $f(2^{\lceil{\log_2(n)}\rceil}-n,\lceil{\log_2(n)}\rceil)=c(n)$ is the difference. The fact that it is an arithmetic progression is attributed to Takashi Tokita by Sloane in the Online Encyclopedia of Integer Sequences (A002487) \cite{oeis-web}, but is not proved there. It is noted and partially proved by Northshield in \cite{northshield2010stern}. We refine and prove it here, stating the progression's starting points explicitly, and use it to extend Algorithm \ref{sb_alg} to the more efficient Algorithm \ref{sb_alg_log}. 
	\begin{theorem}
		\label{Bitflips_arithmetic}
		Let $0<n<2^i$, with $k=\lceil{\log_2(n)}\rceil$.  Then 
		$$
		f(n,i) = f(n,k) + (i-k)\cdot f(2^k-n,k).
		$$
	\end{theorem}
\begin{proof}
	\emph{Case  1.}  $n=2m$ is even. Then $\ceil{\log(m)}=k-1$, and
	\begin{align*}
	f(n,i)&=f(2m,i) \\
	&= f(m,i-1)  &\text{ by Corollary ~\ref{props_fni2}}\\
	&= f(m,k-1)\\
	&\qquad+ ((i-1)-(k-1))f(2^{k-1}-m,k-1)&\text{by induction}&\\
	&= f(2m,k) + (i-k)f(2^k-2m,k)&\text{by Corollary~\ref{props_fni2}}\\
	&= f(n,k) + (i-k)f(2^k-n,k).
	\end{align*}	
	\emph{Case  2.}  $n=2m+1$ is odd, $n\ne 2^{k-1}+1$. Then $\ceil{\log(2m)}=\ceil{\log(2m+2)}=k$, and
	\begin{align*}
	f(n,i)&=f(2m+1,i)\\ 
	&= f(2m,i)+f(2(m+1),i)&\text{by Corollary~\ref{props_fni2}}&\\
	&= f(2m,k)+(i-k)f(2^k-2m,k)
	\\&\qquad+ f(2m+2,k)+(i-k)f(2^k-(2m+2),k)&\text{by Case 1}\\
	&= f(n,k)+(i-k)f(2^k-n,k). &\text{ by Corollary~\ref{props_fni2}}
	\end{align*}	
	\emph{Case  3.}  $n=2^{k-1}+1$. 
	\begin{align*}
	f(n,i)&=f(2^{k-1}+1,i)\\
	&= (k-1)+(i-k)k &\text{by Lemma \ref{Bitflips_2i-1plus1_closed}}&\\
	&= (k-1)+(i-k)f(2^{k-1}-1,k)&\text{by Lemma \ref{Bitflips_2iless1_closed}} \\
	&= (k-1)+(i-k)f(2^k-n,k)\\
	&= f(n,k)+(i-k)f(2^k-n,k). &\text{by Lemma \ref{Bitflips_2i-1plus1_closed}}\\
	\end{align*}	
\end{proof}
Theorem \ref{Bitflips_arithmetic}, with $k=\ceil{\log_2(n)}$, gives  Algorithm \ref{sb_alg_log} for calculating the number of BSD representations of $n$. Algorithm \ref{sb_alg} is $\mathcal{O}(\log(n))$ when calculated on $k=\ceil{\log_2(n)}$ bits, so Algorithm \ref{sb_alg_log} is $\mathcal{O}(\log(n))$ as well.
\begin{algorithm}[H]\caption{Number of BSD representations on $i$ bits for $n$ (logarithmic)}\label{sb_alg_log}
	\begin{algorithmic}[0] 
		\Function{BSD-fast}{$n,i$} 
		\State $n \gets |n|$ 
		\State $k \gets \ceil{\log_2(n)}$
		\State $num\_reps \gets \Call{BSD}{n,k} + (i-k)\times\Call{BSD}{2^k-n,k}$ 		\State \textbf{return} $num\_reps$
		\EndFunction
	\end{algorithmic}
\end{algorithm}

\section{Acknowledgements}
	The author wishes to thank Vanessa Job for useful discussions. 
	The author would also like to thank the anonymous referees for their careful readings and helpful suggestions.

 \section*{Conflict of interest}
The author declares that she has no conflict of interest.

\bibliographystyle{spmpsci}      
\bibliography{stern-diatomic}

\end{document}